\documentclass[a4paper,16pt,reqno]{amsart}

\usepackage{amsmath, amsthm, amscd, amsfonts, amssymb, graphicx, color}
\usepackage{subfig}
\usepackage{subfig}
\usepackage{cite}
\usepackage[framemethod=default]{mdframed}
\usepackage{showexpl}

\newtheorem{hypothesis}{Hypothesis}
\newtheorem{rmk}{Remark}

\begin{document}
	
	 \newcommand{\be}{\begin{equation}}
	 \newcommand{\ee}{\end{equation}}
	 \newcommand{\bt}{\beta}
	 \newcommand{\al}{\alpha}
	 \newcommand{\laa}{\lambda_\alpha}
	 \newcommand{\lab}{\lambda_\beta}
	 \newcommand{\no}{|\Omega|}
	 \newcommand{\nd}{|D|}
	 \newcommand{\Om}{\Omega}
	 \newcommand{\h}{H^1_0(\Omega)}
	 \newcommand{\lt}{L^2(\Omega)}
	 \newcommand{\la}{\lambda}
	 \newcommand{\ro}{\varrho}
	 \newcommand{\cd}{\chi_{D}}
	 \newcommand{\cdc}{\chi_{D^c}}
	 \newtheorem{thm}{Theorem}[section]
	 \newtheorem{cor}[thm]{Corollary}
	 \newtheorem{lem}[thm]{Lemma}
	 \newtheorem{prop}[thm]{Proposition}
	 \theoremstyle{definition}
	 \newtheorem{defn}{Definition}[section]
	 \newtheorem{exam}{Example}[section]
	 \theoremstyle{remark}
	 \newtheorem{rem}{Remark}[section]
	 \numberwithin{equation}{section}
	 \renewcommand{\theequation}{\thesection.\arabic{equation}}
	 \numberwithin{equation}{section}
	 %
	 %
	 %
	\title[ The first eigenvalue   of a nonlinear elliptic system]{ The first eigenvalue and eigenfunction  of a nonlinear elliptic system}
	\author[Bozorgnia, Mohammadi, Vejchodsk\'y]{Farid Bozorgnia, Seyyed Abbas  Mohammadi,  Tom\'a\v{s} Vejchodsk\'y}
	 \address{Farid Bozorgnia, CAMGSD, Instituto Superior T\'{e}cnico, Lisbon, Portugal.}  \email{bozorg@math.ist.utl.pt }
	
	 \address{Seyyed Abbas  Mohammadi, Department of Mathematics, College of Sciences,
	 	Yasouj University, Yasouj, Iran, 75918-74934}
	 \email{mohammadi@yu.ac.ir}
	
	 \address{Tom\'a\v{s} Vejchodsk\'y, Institute of Mathematics, Czech Academy of Sciences,
	 	\v{Z}itn\'a 25, CZ-115\,67, Prague~1, Czech Republic}
	    \email{vejchod@math.cas.cz}

	 \date{\today}

	 \begin{abstract}
	 	In this paper, we study the first eigenvalue  of a  nonlinear elliptic system involving  $p$-Laplacian as the differential
	 	operator. The principal eigenvalue of the system and the corresponding eigenfunction  are investigated both analytically and numerically. An alternative   proof  to show   the simplicity of the first eigenvalue is given. In addition,  an upper and lower bounds of the first eigenvalue  are provided. Then, a  numerical algorithm is developed  to  approximate  the principal eigenvalue.  This  algorithm generates a decreasing  sequence of positive numbers
	 	and various examples numerically indicate its     convergence.  Further, the algorithm
	 	is generalized to a class of gradient quasilinear elliptic systems.
	 \end{abstract}

	 \maketitle

	 \noindent
	 \textbf{Keywords}: nonlinear elliptic system, $p$-Laplacian, eigenvalue problem, simplicity, numerical approximation\\
	 \textbf{2010 MSC}:  35P30, 34L15, 34L16,  35J92
	 \section{Introduction}

	 Nonlinear elliptic eigenvalue problems  form a class of important  problems in the theory and applications of partial differential equations and they have been extensively studied in the past decades by many researchers. In particular,  problems involving $p$-Laplace operator are of great interest and importance from both  the
	 theoretical and applied aspects \cite{anane,  le,lind1,lind2, lind3, moh1,moh2}.
	
	In this paper, we consider a nonlinear elliptic system involving two nonlinear eigenvalue problems where	 the differential operators are  two  $p$-Laplace operators. The system is weakly coupled such  that the two solution components interact through the source terms only.
	
	 Let $\Omega  \subset  \mathbb{R}^N$ be  a  bounded   domain  with smooth  boundary.
	 	 Our aim is to study, both analytically and
	 	 numerically,  the principal eigenvalue denoted by $\lambda(p, q)$  and the corresponding
	 first eigenfunction $(u,v)$   of the following  elliptic eigenvalue system

	 \begin{equation}\label{mainsys}
	 \left\{
	 \begin{array}{lrl}
	 - \Delta_{p} u =\lambda  |u|^{\alpha-1}  |v|^{\beta-1} v  &    \text{in } \Omega, \\
	 - \Delta_{q} v =\lambda  |u|^{\alpha-1} |v|^{\beta-1} u  &   \text{in } \Omega,\\
	 u=v=0 &   \text{on } \partial\Omega,
	 \end{array}
	 \right.
	 \end{equation}
	where $\Delta_{p} u=  \text{div}(| \nabla u|^{p-2} \nabla u)$ is the $p$-Laplacian and $p, q > 1$  and $\alpha, \beta\geq 1$   are real numbers satisfying
	  \begin{equation}\label{alfa}
	 \frac{\alpha}{p}+ \frac{\beta}{q}=1.
	  \end{equation}

	 The  first eigenvalue $ \lambda(p,q)  $  of system \eqref{mainsys}  is defined  as the least positive  parameter $ \lambda $  for
	 which  system  \eqref{mainsys} has a solution $(u,v)$   in
	 $ {W^{1,p}_{0}(\Omega) \times W^{1,q}_0(\Omega)}$  such that  both    $ u \neq 0$ and $v\neq 0$. This  eigenvalue problem  has a variational form which will be explained in the next section.
	
	 The elliptic system  \eqref{mainsys} has been studied in \cite{khalil}  and some close variants of it have  been studied in several works, let us mention  for example \cite{Bob1, Bob2,boccardo,bonder,napoli}.
	In particular, these papers investigate the first eigenvalue, the corresponding eigenfunction, their  existence, uniqueness, positivity, and isolation in bounded or unbounded domains, with various boundary conditions (see, e.g. \cite{bonder} and the references therein).
	 The coupled system    \eqref{mainsys}  arises in different fields
	 of application. For instance,  the case   $ p > 2$   appears in the study of
	 non-Newtonian fluids, pseudoplastics  and the case  $1 < p <2$  in reaction-diffusion problems,
	 flows through porous media, nonlinear elasticity, and glaciology for $p = \frac{4}{3}$, see  \cite{khalil}.

In \cite{NZ} the author studies properties of the positive principal eigenvalue for the following degenerate elliptic system
\begin{equation}\label{Nikolassys}
\left\{
  \begin{array}{lrl}
    - \textrm{div}(\nu_{1}(x)|\nabla u|^{p-2} \nabla u)  =\lambda \, a(x) |u|^{p-2}  u + \lambda \, b(x)  |u|^{\alpha-1} |v|^{\beta-1} v  &    \text{in } \Omega,
    \\
    - \textrm{div}(\nu_{2}(x)|\nabla v|^{p-2} \nabla v)  =  \lambda \, d(x) |v|^{q-2}  v  +     \lambda \, b(x)  |u|^{\alpha-1} |v|^{\beta-1} u  &   \text{in } \Omega,
    \\
    u=v=0 &   \text{on } \partial\Omega,
  \end{array}
\right.
\end{equation}
where  $ p, q > 1 $  and $\alpha, \beta\geq 1$ satisfy \eqref{alfa}.
Note that choosing $ \nu_{1}(x)=  \nu_{2}(x) =1$, $a(x)= d(x)  =0$, and $b(x)=1$ in system  (\ref{Nikolassys}), we obtain system (\ref{mainsys}).
The main result \cite[Theorem 1.1]{NZ} applied to this special case provides the simplicity and isolation of the first eigenvalue of (\ref{mainsys}) and positivity of  corresponding   first eigenfunction. More precisely, it states that
the system (\ref{mainsys}) admits a positive principal eigenvalue $\lambda_{1}$, satisfying
\[
\lambda_{1}= \underset{(u,v)\in L}{\textrm{inf}}  \left[\frac{\alpha}{p} \int_{\Omega} |\nabla u|^{p} \, dx  + \frac{ \beta}{q} \int_{\Omega} |\nabla v|^{q}   \, dx    \right],
\]
where the set $L$ is
\[
L=\left\{(u,v)  \in W^{1,p}_{0}(\Omega)  \times W^{1,q}_{0}(\Omega):  \,  \int_\Omega |u|^{\alpha} |v|^{\beta}\,  dx  =1\right\}.
\]
Furthermore, each component of  the associated normalized  eigenfunction  $(u_1, v_1)$  is nonnegative.
	
We address some analytical aspects of the first eigenvalue and the corresponding eigenfunction of system (\ref{mainsys}) in this paper.
We provide a different proof of the simplicity of $\lambda(p,q)$ which has been first  addressed   in   \cite{khalil}. Moreover,  it is established that system \eqref{mainsys} reduced to the $p$-Laplacian eigenvalue problem when $p=q$. Next, we will derive a lower and upper estimate for the principal eigenvalue of system \eqref{mainsys}.
	
	   Deriving sharp  bounds for eigenvalues of  elliptic systems  is a challenging   problem which  has been investigated by several authors, e.g. \cite{boccardo,bonder,napoli}.

	 In general, the value of  the  first eigenvalue of \eqref{mainsys} is not  explicitly known even for one dimensional problem; but it is important to determine it due to numerous physical applications. However, for the specific case $p=q$, the system is reduced to  the scalar $p$-Laplace eigenvalue problem and  the spectrum is known exactly in dimension one.     	 To this end, we develop a numerical algorithm computing an approximation of the principal eigenvalue.
          The algorithm is robust and efficient for various domains with different values of parameters $p, q, \alpha$ and $\beta$.
	  Moreover, we explain how to generalize it for a large class of  quasilinear elliptic systems.
	  We prove its convergence in the case $ p=q $  where the system  reduces  to   the $p$-Laplace eigenvalue problem.
	  	  	
	 It is worth to mention that the corresponding scalar equation, i.e., the $p$-Laplace  eigenvalue problem,  has been studied intensively from both the analytical and numerical point of view \cite{farid,horak,le,lind1,lind2,lind3}.

	  The paper is organized as follows. In section~2, we  provide  important definitions, recall needed mathematical background and present preliminary results.  In section~3, we  provide an alternative proof of the simplicity of $\lambda(p,q)$ which has been first addressed in   \cite{khalil}.
Further, lower and upper  estimates  for the principal eigenvalue will be obtained in this section. Section~4 describes the numerical algorithm. In section~5, we
provide several numerical examples illustrating the efficiency and applicability of this method.

	 \section{ Mathematical Background}\label{mathback}
	 In this section we provide  the necessary    mathematical background.
	 Let us at first address the scalar $p$-Laplace    eigenvalue problem.

	 The first eigenvalue of  the $p$-Laplace operator in $W^{1,p}_{0}(\Omega),$ denoted by $\lambda (p)$   for $ 1\leq p <\infty $ is   given by
	 \begin{equation}\label{lambdap}
	 \lambda (p)=\underset{  u\neq 0}{\underset{u \in W^{1,p}_{0}(\Omega)}{ \min}} \frac{\int_{\Omega} |\nabla u(x)|^{p} dx }{\int_{\Omega}| u(x)|^{p}dx}.
	 \end{equation}
	
	 The corresponding Euler-Lagrange equation is
	 \begin{equation}\label{plapeq}
	 \text{div}(| \nabla u|^{p-2} \nabla u) + \lambda(p) |u|^{p-2} u = 0.
	 \end{equation}
	
	 Equation  \eqref{plapeq} is interpreted in the usual weak form with test-functions in $W^{1,p}_{0}(\Omega)$:
	\begin{defn}\label{piegen}
	 	A  nonzero function  $u\in W^{1, p}_{0}(\Omega) \cap C(\overline{\Omega}),$ is called a $p$-eigenfunction, if there exists $\lambda(p) \in \mathbb{R}$
	 	such that
	 	\[
	 	\int_{\Omega}|\nabla u|^{p-2} \nabla u \cdot \nabla\phi \, dx=\lambda(p) \int_{\Omega}| u |^{p-2}  u \, \phi \, dx,  \quad \forall \, \phi \in W^{1,
	 		p}_{0}(\Omega).
	 	\]
	\end{defn}
	 The associated number $\lambda(p)$ is called a $p$-eigenvalue.
		 For every $ 1<p< \infty,$ the first (i.e. the smallest) eigenvalue is simple and isolated and the corresponding eigenfunction is a bounded continuous function on $\overline{\Omega}$ which does not change sign \cite{lind1}.

There are two important limit cases;  as $p$ tends to one and infinity.  We   recall   the  result  of    \cite{KF},  which says  that   the  first eigenvalue $\lambda(p)$
converges to the Cheeger constant  $h(\Omega)$  as  $p  \rightarrow 1$.   Furthermore, the associated eigenfunction
converges to the characteristic function $\chi_{C_\Omega}$ of  the  Cheeger set $C_\Omega$, i.e.,  the subset of $\Omega$
 which
minimizes the ratio  $|\partial D| / |D|$  among all simply connected  $ D \subset \Omega$.

  The first eigenvalue $\Lambda_{\infty}$ of the infinity Laplacian corresponds to  the reciprocal value of the radius of the  largest ball that can be inscribed in the domain $\Omega$. More precisely
\[
\Lambda_{\infty} =\frac{1}{\underset  {x \in \Omega}  {\max }\, \text{dist}(x, \partial \Omega)}    =\underset { p \rightarrow \infty }  {\lim}  \lambda(p)^{\frac{1}{p}},
\]
where $\lambda(p)$   is  the first eigenvalue of the $p$-Laplace operator, see \cite{lind1}.
 In addition,  if the   domain is a ball,  then
  the infinity eigenfunction is the distance function $d(x)= \text{dist} (x, \partial \Omega)$.   Obviously,  for a unite ball centered at the origin $d(x)=  1-|x|$  and $ \Lambda_\infty=1$.

Now, we return to the nonlinear system \eqref{mainsys}.
	\begin{defn}
	 	The  first eigenvalue  $\lambda(p,q)$  of \eqref{mainsys}   is defined  as the least  positive  parameter $ \lambda$ for
	 	which   system  \eqref{mainsys}   has  a solution $(u,v)$  in the product Sobolev space $ W^{1,p}_{0}(\Omega) \times W^{1,q}_{0}(\Omega) $    such that both $ u\neq 0$  and $ v\neq  0$.
	\end{defn}
	
	 Here by a solution to \eqref{mainsys} we mean a pair $(u,v)$ in $ W^{1,p}_{0}(\Omega) \times W^{1,q}_{0}(\Omega) $ such that
	 \begin{align}\label{varformofsol}
	  \int_{\Omega}|\nabla u|^{p-2} \nabla u \cdot \nabla\phi \, dx  + 	
	    \int_{\Omega}|\nabla v|^{q-2} \nabla u \cdot \nabla\psi  \, dx = \nonumber
	    \\
	    \lambda\left( \int_{\Omega}|u|^{\alpha-1}  |v|^{\beta-1} v \phi \, dx + \int_{\Omega}|u|^{\alpha-1}  |v|^{\beta-1} u \psi \, dx\right),
	  	 \end{align}
		 \[  \quad \forall \, (\phi, \psi )  \in
	 W^{1,p}_{0}(\Omega) \times  W^{1,q}_{0}(\Omega).
	 \]
      Defining the Rayleigh quotient
      $$\mathcal{R}(u,v)= \frac{\frac{\alpha}{p}\int_{\Omega} |\nabla u(x)|^{p} dx  + \frac{\beta}{q}\int_{\Omega} |\nabla v(x)|^{q} dx,}{\int_{\Omega} |u(x)|^{\alpha-1}|v (x)|^{\beta-1} u(x) v(x) dx},$$
	 the principal eigenvalue  $\lambda(p, q)$   can be variationally characterized by  minimizing the functional $\mathcal{R}$
	 over the set
	 \[
	 \mathcal{A}=\left\{(u, v) \in  W^{1,p}_{0}(\Omega) \times W^{1,q}_{0}(\Omega): \int_{\Omega} |u(x)|^{\alpha-1}|v (x)|^{\beta-1} u(x) v(x) dx  > 0 \right\}.
	 \]
	Thus
	 \begin{equation}\label{lamvar}
	 \lambda(p,q) =\min \left\{ \mathcal{R}(u,v), (u,v) \in \mathcal{A}  \right\}
	 \end{equation}
	and the minimizer is the pair of eigenfunctions $(u,v)$, see \cite{khalil}.
	If  $(u,v)$ is the pair of eigenfunctions corresponding to the first eigenvalue of \eqref{mainsys}, then
$$\mathcal{R}(|u|,|v|) \leq  \mathcal{R}(u,v),$$
because
	 $$\int_\Omega |u|^{\alpha-1} |v|^{\beta-1} u v dx   \leq   \int_\Omega |u|^{\alpha-1} |v|^{\beta-1} |u| |v| dx.$$

	Consequently,  in view of variational formulation \eqref{lamvar}, we deduce that
	 if $(u,v)$ is a minimizer in \eqref{lamvar}, so is $(|u|,|v|)$. Therefore we may assume that $u$ and $v$ are nonnegative. In addition, if the pair $(u,v)$ is a nonnegative  weak solution to \eqref{mainsys}, then  $u,v>0$ in $\Omega$ due to the maximum principle  of V\`{a}zquez \cite{vaz}. As   mentioned in introduction, to see about simplicity of first eigenvalue and positivity of corresponding eigenfunction for general system (\ref{Nikolassys}) we refer to \cite{NZ}.
	%
	
	 		  The existence of a principal eigenvalue, simplicity and the isolation of the first eigenvalue  have been proved  for \eqref{mainsys} and its variants    in \cite{boccardo,bonder, khalil, Pezzo, napoli}. Let us recall that the first eigenvalue $\lambda(p,q)$   of \eqref{mainsys} is simple if for any two pairs of  corresponding eigenfunctions  $(u, v)$   and $( \phi, \psi )$ there exist real numbers $k_1$ and $k_2$ such that $u=k_1\phi$ and $v = k_2\psi$.

	 	\section{Analytical results}
	 	
	 	In this section we examine certain   analytical aspects of system \eqref{mainsys}.	 Simplicity of the principal eigenvalue is one of its main features and it has been investigated in \cite{khalil}.  Here we  provide an alternative proof which is more straightforward and based on the proof given by Belloni and Kahwol \cite{BK}  establishing  the simplicity of the principal eigenvalue of  scalar problem  \eqref{plapeq}.
	
	 \begin{thm} The first eigenvalue of system  \eqref{mainsys}  is  simple.
	 \end{thm}
	
	 \begin{proof}
	 	Let  $(u,v) $ and  $(\phi,\psi)$   be two pairs of    eigenfunctions   associate with  $\lambda(p,q)$.  As we mentioned above, we assume that $u,v >0$ and $\phi, \psi>0$  in
	 	$\Omega$. Without loss of generality, we assume that these eigenfunctions are normalized such that
	 	\[
	 	\int_{\Omega} u ^{\alpha} \, v ^{\beta}\,dx=\int_{\Omega} \phi^{\alpha} \, \psi^{\beta}\,dx=1.
	 	\]
	 	 We show 	 	that there exist real numbers $k_1, k_2$ such that
	 	$ u=k_1 \phi $ and  $ v=k_2 \psi $.	 	
	 	  Note that
	 	\[
	 	w_1=\left(\frac{u^{p}+\phi^{p}}{2}\right)^{\frac{1}{p}} \quad \text{and} \quad   w_2=\left(\frac{v^{q}+\psi^{q}}{2}\right)^{\frac{1}{q}}
	 	\]
	 	 are admissible functions which means they belong to $\mathcal{A}$. 
	 	 In view of  variational form \eqref{lamvar}, we observe that
	 	 \begin{equation}\label{f23}
	 	 \lambda(p,q)\le \frac{\frac{\alpha}{p}\int_{\Omega} |\nabla w_1|^{p} dx  + \frac{\beta}{q}\int_{\Omega}
	 	 	|\nabla w_2|^{q} dx}{\int_{\Omega}w_{1}^{\alpha} w_{2}^{\beta} \, dx}.
	 	 \end{equation}
	 	
	 	 We show that
	 	 \begin{equation}\label{w12ineq}
	 	 w_1^\alpha w_2^\beta= \left(\frac{u^{p}+\phi^{p}}{2}\right)^{\frac{\alpha}{p}} \,  \left(\frac{v^{q}+\psi^{q}}{2}\right)^{\frac{\beta}{q}}\geq \frac{1}{2}(u^\alpha v^\beta + \phi^\alpha \psi^\beta).
	 	 \end{equation}
	 	 Due to the H\"{o}lder's inequality for counting measure, we observe that
	 	 $$(u^\alpha, \phi^\alpha) \cdot (v^\beta, \psi^\beta)\leq \left(u^{\frac{\alpha p}{\alpha}}+\phi^{\frac{\alpha p}{\alpha}} \right)^{\frac{\alpha}{p}} \left(v^{\frac{\beta q}{\beta}} +\psi^{\frac{\beta q}{\beta}} \right)^{\frac{\beta}{q}},$$
	 	 which yields \eqref{w12ineq}.
	 	 Thus,  (\ref{f23}) and (\ref{w12ineq}), and the normalization of $(u,v)$ and  $(\phi,\psi)$  yields
	 	
	 	\begin{equation}\label{f124}
	 	\lambda(p,q)\le  \frac{\alpha}{p} \int_{\Omega} |\nabla  w_1|^{p}\, dx     + \frac{\beta}{q}   \int_{\Omega} |\nabla  w_2|^{q} \, dx.	 	\end{equation}

	 For gradients of $w_1$  and $w_2$ 	we have
	 	\[
	 	|\nabla w_1|^{p}= \left(\frac{u^{p}+\phi^{p}}{2}\right) \left |\frac{ u^{p} \nabla \log \, u + \phi^{p}  \nabla \log\, \phi}{u^{p}+\phi^{p}} \right|^{p},
	 	\]
	 	\[
	 	|\nabla w_2|^{q}= \left(\frac{v^{q}+\psi^{q}}{2}\right) \left |\frac{ v^{q} \nabla \log \, v + \psi^{q}  \nabla \log\, \psi}{v^{q}+\psi^{q}} \right|^{q}.
	 	\]
	 Recalling Jensen's inequality
	 	\[
	 	\theta\left(\frac{\Sigma a_i x_i}{\Sigma a_i}\right) \le  \frac{\Sigma a_i \theta(x_i)}{\Sigma a_i}
	 	\]
	 	 for convex function $  \theta(\cdot)= |\cdot |^{p}$,
	   we obtain by choosing	 	
	 	\begin{equation*}
	 	\left \{
	 	\begin{array}{ll}
	 	a_1=  \frac{u^{p}}{u^{p}+\phi^{p}}, \quad  a_2= \frac{\phi^{p}}{u^{p}+\phi^{p}},\\
	 	x_1= \nabla \log\, u, \quad x_2= \nabla \log\, \phi,
	 	\end{array}
	 	\right.
	 	\end{equation*}
	 	 the following inequalities:
	 	\[
	 	|\nabla w_1|^{p} \le \frac{1}{2} |\nabla  u|^{p} + \frac{1}{2}|\nabla  \phi|^{p},
	 	\]
	 	\[
	 	|\nabla w_2|^{q} \le \frac{1}{2} |\nabla  v|^{q} + \frac{1}{2}|\nabla  \psi|^{q}.
	 	\]
	 	These inequalities  are strict at points where
 \[
 \nabla \log  u(x)\neq \nabla \log  \phi(x) \quad  \text{and} \quad  \nabla \log \,  v(x)\neq \nabla \log \, \psi(x).
 \]
 	 Therefore, we assume for the moment  that  $ \nabla \log \,  u\neq \nabla \log \, \phi$ or
	 	  $ \nabla \log \,  v\neq \nabla \log \, \psi $ in a set
	 	  of positive measure. Consequently, inequality (\ref{f124}) implies
	 	\begin{equation}\label{f126}
	 	\lambda(p,q)  <   \frac{\alpha}{2p} \int_{\Omega}(  |\nabla  u|^{p}  + |\nabla  \phi|^{p}) \, dx     + \frac{\beta}{2q}   \int_{\Omega} ( |\nabla  v|^{q}  + |\nabla  \psi|^{q}) \, dx= \lambda(p,q),
	 	\end{equation}	
where the last equality follows from  \eqref{lamvar} and the fact that $(u,v) $ and $(\phi,   \psi)$  are normalized.  Contradiction
\eqref{f126} shows that
\[
 \nabla \log    u = \nabla \log \, \phi  \quad \text{and}  \quad    \nabla \log   v   =  \nabla \log  \psi \quad    \text{a.e. in} \quad   \Omega.
\]
Therefore there exist constants $k_1$ and $k_2$ such that  $ u =k_1 \phi $ and $ v= k_2 \psi$.

	 \end{proof}

	 One interesting feature of  system \eqref{mainsys} is  that it will  reduce to the scalar equation \eqref{plapeq}  with Dirichlet  boundary conditions for $p=q$.
	 \begin{thm}\label{p=q}
	 	Let $p=q$ and   $(u,v)$ be  a solution of \eqref{mainsys}. Then $u$ equals $v$ in $\Omega$ for all $\alpha,\beta \geq 1$ satisfying (\ref{alfa}).
	 	Moreover, function $u=v$ solves \eqref{plapeq}.
	 \end{thm}
	 \begin{proof}
	 	Suppose that $u \neq v$ in $\Omega$. Then, without loss of generality, there  is a  subset  $D$ of $\Omega$  of positive measure  such that
	 	$$D=\{x\in \Omega: u(x)<v(x)\}.$$
	 		 	The set $D$ is an open set  due to the fact that $u,v \in C^1({\Omega})$, see \cite{khalil}.
	  We define
	  	\begin{equation*}
	  \eta(x)=  	\left \{
		\begin{array}{ll}
	  v(x)-u(x)  &   \text{in } D, \\
	  0     &  \text{in } \Omega\setminus D.
	  \end{array}
	  \right.
	  \end{equation*}
	 This $\eta$ belongs to $W^{1,p}(\Omega)$.  Considering $\eta$ as a test function, variational formulation   \eqref{varformofsol} with $p=q $ and  $ \lambda= \lambda(p,q) $  yields
	 	\begin{align*}
	 	\int_{D}|\nabla u|^{p-2} \nabla u \cdot \nabla\eta \, dx= \lambda \int_{D} u^{\alpha-1}v^\beta \eta dx,
	 	\\
	 	\int_{D}|\nabla v|^{p-2} \nabla v \cdot \nabla\eta \, dx= \lambda \int_{D} u^{\alpha}v^{\beta-1} \eta dx,
	 	\end{align*}
	 	and consequently,  we have
	 	\begin{equation}\label{subeq}
	 	\int_{D}(|\nabla v|^{p-2} \nabla v-|\nabla u|^{p-2} \nabla u)\cdot \nabla (u-v) \, dx=  \lambda \int_{D}  u^{\alpha-1}v^{\beta-1}(u-v)(v-u)  \, dx.
	 	\end{equation}
	 In view of  the positivity of the eigenfunctions of \eqref{mainsys}, the right hand side of \eqref{subeq} is negative. Recalling
	 	the inequality from \cite{lind3}:
	 $$	\left<|b|^{p-2}b-|a|^{p-2}a, b-a\right>>0,\:\; \forall  a,  b \in \mathbb{R}^N, \,  a\neq b,\:\; \text{and}\:\; p>1,$$
	 	where $\left< \cdot, \cdot \right>$ denotes the inner product in $ \mathbb{R}^N$ and setting $b=\nabla v$ and $a=\nabla u$, we deduce  that the left hand side of  \eqref{subeq} is a  positive quantity.
	 	This is a contradiction  with  the negativity of the  right hand side and, thus, $u=v$ in $\Omega$.
	 \end{proof}
	
	 Now, we will provide some upper and lower bounds for the first eigenvalue  of system \eqref{mainsys}. Such estimates for eigenvalues have been considered for similar systems in various papers, for example \cite{boccardo,bonder,napoli}. First, we prove the lower  bound.
	 \begin{thm}\label{lowerest}
	 	Let $ \lambda(p,q) $  be defined by \eqref{lamvar}, then
	 	$$ \lambda(p,q)\geq \min \{\lambda(p),\lambda(q)\}, $$
	 	where $\lambda(p)$ and $\lambda(q)$ are the principal eigenvalues of $p$-Laplacian given by \eqref{lambdap}.
	 \end{thm}
	 \begin{proof}
	 	Let  $(u,v) $  be the  normalized  eigenfunction   associated with  $ \lambda(p,q) $. Using Young's inequality, we obtain
	 	$$ \lambda(p,q)=\frac{\frac{\alpha}{p}\int_{\Omega} |\nabla u|^{p} dx  + \frac{\beta}{q}\int_{\Omega}
	 		|\nabla v|^{q} dx}{\int_{\Omega}u^{\alpha} v^{\beta} \, dx}\geq  \frac{\frac{\alpha}{p}\int_{\Omega} |\nabla u|^{p} dx  + \frac{\beta}{q}\int_{\Omega}
	 		|\nabla v|^{q} dx}{\frac{\alpha}{p}\int_{\Omega}u^{p}dx + \frac{\beta}{q}\int_{\Omega}v^{q}dx}.$$
                Applying inequality
                $$
                  \frac{a+b}{c+d} \geq \min \left\{ \frac{a}{c}, \frac{b}{d} \right\}
                $$
                with $a=\frac{\alpha}{p}\int_{\Omega} |\nabla u|^{p} dx$,
                $b=\frac{\beta}{q}\int_{\Omega} |\nabla v|^{q} dx$, etc., we obtain
%
%
	 	$$\lambda(p,q)\geq \min\left\{\frac{\int_{\Omega} |\nabla u|^{p} dx }{\int_{\Omega}u^{p}dx},\: \frac{\int_{\Omega} |\nabla v|^{q} dx }{\int_{\Omega}v^{q}dx} \right\}\geq \min \{\lambda(p),\lambda(q)\}.$$
	 \end{proof}
	

     Determining an upper bound  for the first eigenvalue is a more subtle problem and we will investigate it for certain  special cases.
     First, we address this question in dimension $N=1$.  The upper bound derived below is based on the following theorem from \cite{borell}.
     \begin{thm}\label{borell}
     	Let $a_1,\dots , a_n$ be real numbers all greater or equal  to 1. Suppose that
     	\begin{equation*}
     	\min_{{I\subseteq \{1,\dots,n\}}} \left| \sum_{k\in I} a_k-\frac{1}{2} \sum_{k=1}^n a_k\right|,
     	\end{equation*}
     	is attained for $I=I_0$. Set
     	$$a_{0}=\underset{k\in  I_0}{\sum} a_k, \qquad a_{00}=\underset{ k\notin I_0}{\sum} a_k.$$
     	Let $g_1,\dots ,g_n$ be nonnegative functions on the interval $(0,1)$ such that the function $g_1^{1/a_1},\dots ,g_n^{1/a_n}$ are concave and let $p_1,\dots,p_n$ be real numbers greater or equal to 1. Then
     	$$\int_0^1 \prod_{k=1}^n g_k(x)dx \geq C \prod_{k=1}^n \left(\int_0^1  g_k^{p_k}(x)dx \right)^{1/p_k}, $$
     	where
     	$$C=\left(\prod_{k=1}^n(1+a_kp_k)^{1/p_k}\right)\mathcal{B}(1+a_0,1+a_{00}),$$
     	and $\mathcal{B}$ stands for the beta function
     	$$ \mathcal{B}(r,s)=\int_{0}^{1} x^{r-1}(1-x)^{s-1}dx.$$
     \end{thm}
     Now we are prepared to prove the following theorem.
     \begin{thm}\label{oneup}
     	
     	Let $\Omega=(0,1)$ then for the principal eigenvalue of \eqref{mainsys} we have
     	\begin{equation}\label{sun1}
     	\lambda(p,q)\leq \frac{1}{C}\left(\frac{\alpha}{p}\lambda(p)+\frac{\beta}{q} \lambda(q)\right),
     	\end{equation}
     	where
     	$$C=(1+p)^{\alpha/p}(1+q)^{\beta/q} \mathcal{B}(1+\alpha,1+\beta).$$
     	
     \end{thm}
     \begin{proof}
     	Let $u = u(p)$ and $v = u(q)$ be the eigenfunctions associate with $\lambda(p)$ and $\lambda(q)$ respectively and let they be  normalized such that $\|u\|_{L^p(\Omega)}=\|v\|_{L^q(\Omega)}=1.$
     	In Theorem \ref{borell}, we set
     	$$a_1=\alpha,\; a_2=\beta,\;  g_1=u^\alpha,\; g_2=v^\beta,\;p_1=\frac{p}{\alpha}, \;p_2=\frac{q}{\beta}.$$
     	From here
     	\[
     	a_0=\alpha,\; a_{00}=\beta.
     	\]
     	Recall that $\alpha, \beta \geq 1$ and also $\frac{p}{\alpha}, \frac{q}{\beta}\geq 1$  regarding the fact that $\frac{\alpha}{p}+ \frac{\beta}{q}=1$. It is known that the first eigenfunction of \eqref{plapeq} is concave for one dimensional problems \cite{lind1}.
     	Hence, functions $g_1, g_2$    	are concave as well. Thus, in view of Theorem \ref{borell} 	we observe that
     	\begin{equation}\label{Cineq}
     	\int_{0}^{1}g_1g_2dx=\int_{0}^{1} u^\alpha v^\beta dx\geq C  \|u\|_{L^p(\Omega)}^{\alpha}\|v\|_{L^q(\Omega)}^{\beta}=C,
     	\end{equation}
     	where
     	$$C=(1+p)^{\alpha/p}(1+q)^{\beta/q} \mathcal{B}(1+\alpha,1+\beta).$$
     	
     	Applying the  variational characterization \eqref{lamvar}, we infer that
     	$$\lambda(p,q) \leq  \frac{\frac{\alpha}{p}\int_{\Omega} |\nabla u(x)|^{p} dx  + \frac{\beta}{q}\int_{\Omega} |\nabla v(x)|^{q} dx,}{\int_{\Omega} u^\alpha v^\beta dx}=\frac{\frac{\alpha}{p} \lambda(p)+ \frac{\beta}{q}\lambda(q)}{\int_{0}^{1} u^\alpha v^\beta dx},$$
     	and then employing \eqref{Cineq} we obtain (\ref{sun1}).

     \end{proof}
     The above proof is strongly based upon the concavity of the first eigenfunctions of \eqref{plapeq} in  dimension one. It is worth noting	that the first eigenfunction of \eqref{plapeq}  is  not  concave, in higher dimensions in general \cite{lind1}.
     Note that another  upper bound for dimension one is given in \cite[Section 5]{bonder2}.

     Concerning two dimensions, we obtain an upper bound for the first eigenvalue when $\alpha=\beta=1$,
     provided the following hypothesis holds true.

\begin{hypothesis}\label{hy:1}
If $u(p)$ denotes the eigenfunction corresponding to the first eigenvalue of the scalar $p$-Laplacian \eqref{plapeq} normalized such that $\| u(p) \|_{L^p(\Omega)} = 1$ then
\begin{equation}
  \label{eq:hyp}
  \int_\Omega u(p) u(q) \, dx \geq \int_\Omega u(1) u(\infty) \, dx
\end{equation}
for all $p,q\in[1,\infty]$ satisfying $1/p + 1/q = 1$.
\end{hypothesis}

This hypothesis can be investigated by introducing function
\begin{equation}
  \label{eq:fp}
   f(p)  =   \int_\Omega u(p) u\left(\frac{p}{p-1}\right) \, dx
 \quad\text{for }p\in[1,\infty],
\end{equation}
where the value of $p/(p-1)$ for $p=1$ and $p=\infty$ is understand to be $\infty$ and $1$, respectively.
Note that function $f(p)$ is symmetric in the sense
$$
   f(p) = f\left(\frac{p}{p-1}\right)
   \quad\text{for all }p\in[1,\infty].
$$
Consequently, it is sufficient to investigate $f(p)$ for $p\in[2,\infty]$ only.
Hypothesis \eqref{eq:hyp} is equivalent to the statement $f(p) \geq f(\infty)$ for all $p \in [2,\infty]$.

It is easy to show that the maximum of $f$ is attained at 2. Indeed,
$$
  f(p) = \int_\Omega u(p) u(q) \leq \| u(p) \|_{L^p(\Omega)} \| u(q) \|_{L^q(\Omega)} = 1
\quad\text{and}\quad
  f(2) = 
    \| u(2) \|_{L^2(\Omega)}^2 = 1.
$$
Thus, if $f(p)$ were nonincreasing for $p \in [2,\infty]$ then the minimum of $f(p)$ would be attained at $p=\infty$ and Hypothesis~\ref{hy:1} would hold true. However, the monotonicity of $f(p)$ is not clear.

One possibility how to investigate it is to show the existence and nonnegativity of the derivative
$$
   f'(p) = \frac{-1}{(p-1)^2} f'\left(\frac{p}{p-1}\right)
$$
for $p \in [2,\infty]$.
Figure~\ref{fi:fp} shows numerically computed values of $f(p)$ for interval $\Omega = (0,1)$.
These results indicate that $f(p)$ is smooth, its derivative is positive in $[1,2]$, negative in $[2,\infty]$, $f'(2)=0$,
and consequently that Hypothesis 1 holds true.
Note that in this case $u(1) = \chi_\Omega$ is the characteristic function of $\Omega$, $u(\infty) = 1-|2x-1|$ is the distance function, and hence $f(1) = f(\infty) = 1/2$.
\begin{figure}[ht]
  \centering
  \includegraphics[width=0.8\textwidth]{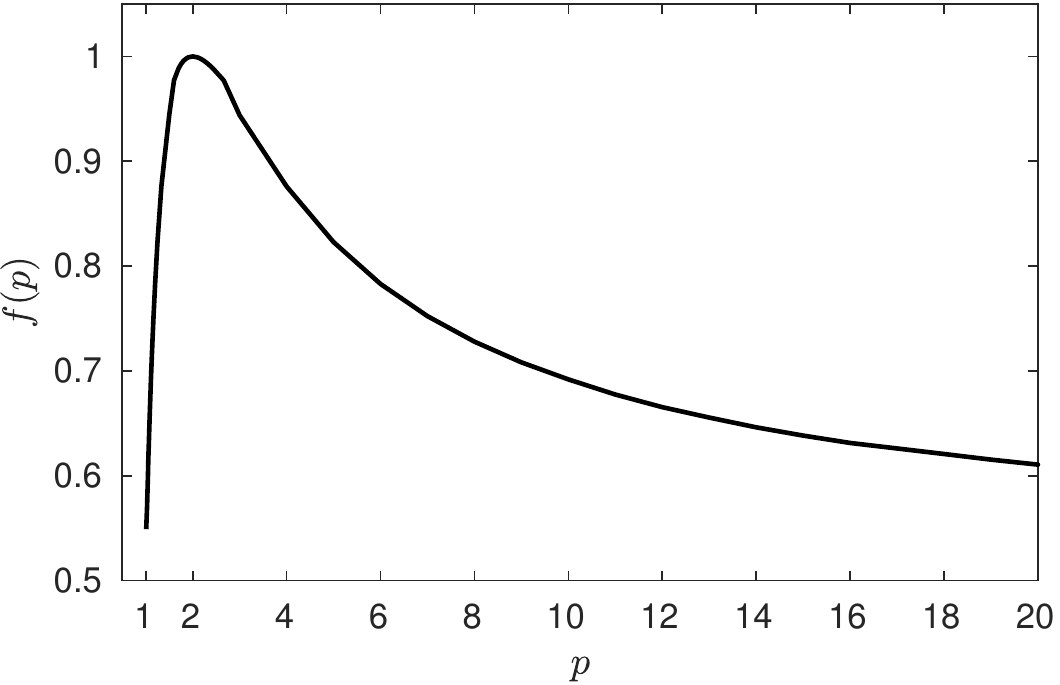}
  \caption{The graph of function $f(p)$ defined in \eqref{eq:fp} for $\Omega = (0,1)$.
  }
  \label{fi:fp}
\end{figure}


    \begin{thm}\label{hiup}
        Assume  Hypothesis~\ref{hy:1} holds true. Let $\Omega$ be a convex subset of $\mathbb{R}^2$ and
    	let $\alpha=\beta=1$.
    	Then for the principal eigenvalue of \eqref{mainsys} we have
    	
    	\begin{equation}
    	  \label{eq:upperbound}
    	  \lambda(p,q) \leq \frac{3\,|\Omega|}{|B_\varrho|} \left(\frac{1}{p}\lambda(p)+\frac{1}{q} \lambda(q)\right),
        \end{equation}
%
    	where
    	$B_{\varrho}$ denotes the largest disc that can be inscribed in $\Omega$ and
    	$\lambda(p), \lambda(q)$ are the principal eigenvalues given by   \eqref{lambdap} and corresponding to $p, q$, respectively.
    \end{thm}

    \begin{proof}
    Considering eigenfunctions $u = u(p)$ and $v = u(q)$ as in the proof of Theorem~\ref{oneup} and using the  variational characterization \eqref{lamvar}, we observe that
    $$
      \lambda(p,q) \leq  \frac{\frac{1}{p}\int_{\Omega} |\nabla u(x)|^{p} dx  + \frac{1}{q}\int_{\Omega} |\nabla v(x)|^{q} dx,}{\int_{\Omega} u v dx}=\frac{\frac{1}{p} \lambda(p)+ \frac{1}{q}\lambda(q)}{\int_{\Omega} u v dx},
    $$
    The need lower bound for $\int_{\Omega} u v \, dx$ is provided by Hypothesis~\ref{hy:1}:
    $$
      \int_{\Omega} u(1) u(\infty) \,dx \leq  \int_{\Omega} u v \,dx,
    $$
    where $ u(1) $ and $ u(\infty)  $ are normalized in  $L^{1}(\Omega)$ and   $L^{\infty}(\Omega)$, respectively.  We know that
    $$
      \int_{\Omega} u(1) u(\infty) \,  dx =  \frac{1}{|C_{\Omega}|} \int_{\Omega} \chi_{C_{\Omega}} \, u(\infty)  \, dx= \frac{1}{|C_{\Omega}|} \int_{C_\Omega}  u(\infty)  \, dx.
    $$

    Further, let $B_\varrho$ be a largest disc inscribed to $\Omega$ and let $\varrho$ be its radius. The normalized eigenfunction of the $\infty$-Laplacian in the disc $B_\varrho$ is $d_B/\varrho$, where
    \[
      d_B(x) = \text{dist}(x, \partial B_{\varrho}).
    \]
    If we extend $d_B$ by zero then
    $$
      u(\infty) \geq \frac{1}{\varrho} d_B \quad\text{in }\Omega,
    $$
    because $B_\varrho \subset \Omega$.
    Thus,
    \[
    \label{eq:est1}
    \frac{1}{|C_{\Omega}|} \int_{C_\Omega}  u(\infty)  \, dx
    \geq \frac{1}{\varrho\,|C_{\Omega}| } \int_{C_\Omega} d_B(x) \,dx
    \geq \frac{1}{\varrho\,|C_{\Omega}| } \int_{B_\varrho} d_B(x) \,dx
    = \frac{|B_\varrho|}{3\,|C_{\Omega}| }
    \geq \frac{|B_\varrho|}{3\,|\Omega|},
    \]
    where we use the fact that $B_\varrho \subset C_{\Omega}$.
    This inclusion follows from \cite[Theorem~1]{KawLac2006}, where the convexity of $\Omega$ is assumed.
    This theorem states that there exists $t^*>0$ such that $C_\Omega = \Omega^{t^*} + t^* B_1$, where $\Omega^{t^*} = \{ x \in \Omega : \operatorname{dist}(x,\partial\Omega) > t^*\}$, $B_1$ is the unit disc and the addition is the Minkowski addition of sets, i.e., $A+B = \{a + b : a\in A,\ b\in B\}$. Shifting $\Omega$ such that the center of $B_\varrho$ is at origin, we immediately see that $(\varrho-t^*)B_1 \subset \Omega^{t^*}$. Consequently, $C_\Omega \supset (\varrho-t^*)B_1 + t^* B_1 = B_\varrho$.

    To conclude, we obtained
    $$
      \int_{\Omega} u v dx \geq \frac{|B_\varrho|}{3\,|\Omega|}
    $$
    and the proof is finished.
\end{proof}

\begin{rem}
 If  the domain is a ball  in  $\mathbb{R}^N $ or a domain where the first eigenfunction of the infinity Laplace operator is the distance function,  then it is easy to see that
\[
\int_{\Omega} u(1) v(\infty) \,  dx = \frac{1}{N+1}.
\]
This yields (under Hypothesis 1) upper bound
$$
  \lambda(p,q) \leq (N+1) \left(\frac{1}{p}\lambda(p)+\frac{1}{q} \lambda(q)\right).
$$
Interestingly, estimate \eqref{eq:upperbound} turns to this bound if $\Omega$ is chosen as a disc in Theorem~\ref{hiup}.
\end{rem}


\section{An algorithm to approximate   the first eigenvalue and  the first eigenfunction}\label{numeralgo}
Algorithm 1 computes an approximation of the first eigenvalue and the corresponding eigenfunction of \eqref{mainsys}.
\\
\\
\\
\begin{mdframed}
	\begin{center}
		\textbf{Algorithm 1}
	\end{center}

\begin{enumerate}
\item
  Set $k=0$ and choose an initial guess $(u, v) \in W^{1,p}_{0}(\Omega) \times  W^{1,q}_{0}(\Omega)$ such that  $u, v  \geq 0$.
\item
  Given $u$, $v$, normalize them as
  \[
    u_{k}=\frac{u}{ ( \int_{\Omega} u^{\alpha} v ^{\beta}  dx)^{\frac{1}{p} } },\quad v_{k}=\frac{v}{ ( \int_{\Omega} u^{\alpha} v ^{\beta}  dx)^{\frac{1}{q}}},
  \]
  and calculate
  $$
    \lambda^{k}=   \frac{\alpha}{p}\int_{\Omega} |\nabla u_{k} |^{p} dx  + \frac{\beta}{q}\int_{\Omega} |\nabla v_{k}|^{q} dx.
  $$
\item
  If $k \geq 1$ and $ |\lambda^{k}-\lambda^{k-1}| < \varepsilon$, then stop.
\item
  Otherwise, solve the following decoupled systems:
  \begin{alignat}{2} \label{11e}
  -\Delta_{p} u &= \lambda^{k}  u_{k}^{\alpha-1}  v_{k}^{\beta} &\quad &\text{in } \Omega,
  \nonumber
  \\
  -\Delta_{q} v &=  \lambda^{k}  u_{k}^{\alpha} v_{k}^{\beta-1} &\quad &\text{in } \Omega,
  \\
  u&=v=0 &\quad &\text{on }\partial\Omega,
  \nonumber
  \end{alignat}
  set $k=k+1$, and go to the  step  (2).
\end{enumerate}
\end{mdframed}
\[
\]

Algorithm 1 computes in every iteration an approximation $\lambda^k$ of the principal eigenvalue
$\lambda(p,q)$ of \eqref{mainsys}. The computed pair of functions $(u_k, v_k)$ approximates the corresponding pair of eigenfunctions. Note that functions $(u_k, v_k)$ are
normalized in every iteration such that  $\int_{\Omega} u_{k}^{\alpha}v_{k}^{\beta} dx=1$.
The algorithm stops, when the distance between two successive approximate eigenvalues is less than a given tolerance $\varepsilon.$

\begin{rmk}
	In view of  Theorem \ref{p=q}, we know that \eqref{mainsys}  is reduced to  scalar equation \eqref{plapeq}  with Dirichlet  boundary conditions when $p=q$. Algorithm 1 in this case reduces to  the algorithm  developed by the first author in \cite{farid} where the   convergence of the iterative scheme to
	the first eigenfunction and the related eigenvalue  has been shown.
\end{rmk}

\begin{rmk}
	In \cite{RE}  two  methods for approximate minimizers of the abstract Rayleigh
	quotient $\frac{\Phi(u)}{\| u\|^{p}}$ have been presented. The functional
        $\Phi$ is assumed there to be strictly convex on a Banach space with
	norm $\|\cdot\|$ and  positively homogeneous of degree $p \in (0,   \infty)$.
	These methods, however,  are not applicable  to calculate the principal eigenvalue of
	\eqref{mainsys} since $ \mathcal{R}( t^{1/p} u, t^{1/q} v)= \mathcal{R}( u,  v). $ 	 	
\end{rmk}

Now we  discuss  the interesting possibility of extending Algorithm 1 to  a class of  quasilinear elliptic systems called gradient systems. These systems have been studied widely in the past decade, see \cite{Arruda} and the reference therein.
Gradient systems are of the following general form:	
\begin{equation}\label{f6}
\left\{
\begin{array}{lrl}
- \Delta_{p} u =\lambda     F_{u}(x,u,v) &    \text{in } \Omega, \\
- \Delta_{q} v =\lambda   F_{v}(x,u,v)   &   \text{in } \Omega,\\
u=v=0     &   \text{on }   \partial\Omega,\\
\end{array}
\right.
\end{equation}
where  $1 < p, q < \infty $ and {$F_u,\:F_v$ denotes  partial derivatives}. The nonlinearity $F: \Omega\times \mathbb{R} \times \mathbb{R} \rightarrow  \mathbb{R}$
is a $C^1$- function,	 satisfying $F(x, 0, 0) = 0$ and
\begin{itemize}
	\item   $F(x, t, s) = F(x, 0, s),\quad \forall\: x\in \Omega,\quad s\in\mathbb{R}\quad\text{and}\quad t\leq 0,	 	$
	\item  $F(x, t, s) = F(x, t, 0),\quad \forall \: x\in \Omega,\quad t\in\mathbb{R}\quad\text{and}\quad s\leq 0,	 	$
	\item   $\left| F_{t}(x, t,s)\right| \le c\left(1+ |t|^{p-1} + |s|^{q\frac{p-1}{p}} \right),$
	\item   $ \left| F_{s}(x, t,s) \right| \le  c\left(1+ |s|^{q-1} + |t|^{p \frac{q-1}{q}} \right)$
\end{itemize}
for all $(x,s,t)\in \Omega \times \mathbb{R} \times \mathbb{R}$.
Under these growth conditions on $F$, the principal eigenvalue is the minimum value of the following Rayleigh quotient
$$
\mathcal{R}(u,v)= \frac{\frac{1}{p}\int_{\Omega} |\nabla u(x)|^{p} dx  + \frac{1}{q}\int_{\Omega} |\nabla v(x)|^{q} dx}{\int_{\Omega} F(x,u(x), v(x)) dx},
$$
over the set
\[
\mathcal{A}={\{(u, v) \in  W^{1,p}_{0}(\Omega) \times  W^{1,q}_{0}(\Omega): \int_{\Omega} F(x,u(x), v(x)) \,  dx {>} 0 }\}.
\]

Now, Algorithm 1  can be easily modified to compute the principal eigenvalue of \eqref{f6}. To this aim, we just  replace  decoupled system \eqref{11e} by decoupled system
\begin{equation}\label{quasisysalg}
\left \{
\begin{array}{lrl}
- \Delta_{p} u =\lambda^{k} F_u(x,u_{k},v_{k})    &    \text{in } \Omega, \\
- \Delta_{q} v =\lambda^{k}  F_v(x,u_{k},v_{k})     &   \text{in } \Omega,\\
u=v=0	 																												&   \text{on } \partial \Omega,
\end{array}
\right.
\end{equation}
and {normalize its}  solution {as}
\[
u_{k+1}=\frac{u}{ \left( \int_{\Omega} F(x,u(x), v(x)) \,  dx \right)^{\frac{1}{p} } },
\]

\[
v_{k+1}=\frac{v}{ \left( \int_{\Omega} F(x,u(x), v(x)) \,  dx \right)^{\frac{1}{q} }    }.
\]
{Performed} numerical tests {indicate} that the extended algorithm  is convergent and efficiently approximates the {principal} eigenvalue.

	 \section{Numerical implementation}
	 This section provides several
	 examples in order to illustrate the efficiency of Algorithm~1.
	 At iteration $k$ we  solve decoupled nonlinear elliptic system \eqref{11e} by the finite element method with piecewise
	 linear basis functions. The resulting discrete system of nonlinear equations is solved by a modified Newton-Raphson method. Using a tolerance on the level of the machine precision and a suitable initial approximation, the Newton-Raphson method converges in at most $50$ iterations for all examples below.
	 Similarly, the tolerance in the fourth step of Algorithm 1 was chosen as $ \varepsilon=5 \times 10^{-5}$ and in all the following examples the algorithm converges in less than 10 iterations.
	
	 \begin{exam}\label{circle}
	Let   $ \Omega$   be  the unit disc  centred at origin.  The radial symmetry then enables us to use polar coordinates $u = u(r),\: v=v(r)$, $0< r < 1$ and to transform system \eqref{mainsys} to one dimensional system
	 \begin{equation}\label{mainsyspolar}
	 \left\{
	 \begin{array}{lrl}
	 - (r|u^\prime|^{p-2}u^\prime)^\prime =\lambda r |u|^{\alpha-1}  |v|^{\beta-1} v  &    \text{in } (0,1), \\
	 - (r|v^\prime|^{q-2}v^\prime)^\prime =\lambda   r |u|^{\alpha-1} |v|^{\beta-1} u  &   \text{in } (0,1),\\
	 u^\prime(0)=v^\prime(0)=0\quad u(1)=v(1)=0.
	 \end{array}
	 \right.
	 \end{equation}
	 Note that all finite element computations in the interval $(0,1)$ are performed with 500 elements (subintervals) of the same length.
	
	Let us start with a simple test case $p=q=2$. For this choice and arbitrary values of $\alpha$ and $\beta$ satisfying \eqref{alfa}, system \eqref{mainsys} reduces to a scalar  eigenvalue problem  for the standard Laplace operator.
	The principal eigenvalue of Laplacian in the unit disc with zero Dirichlet boundary conditions is the square of the first zero of the Bessel function $J_0$, i.e., $\lambda_1 \approx 5.7832$.
	The sequence $\{\lambda^k\}$ computed by Algorithm~1 applied to system \eqref{mainsyspolar} converges to this value and Table~\ref{lapeig} illustrates the speed of this convergence for initial guess $u_0 = v_0 = (1-r)^2$ and parameter values
	$\alpha = 1$ and $\beta = 1$.

	Similarly, both sequences $\{u^k\}$ and $\{v^k\}$ computed by Algorithm~1 converge to the  normalized first eigenfunction of the Laplacian
	$$w_1(r) = J_0\left(\lambda_1^{1/2} r \right)/ \left\| J_0\left(\lambda_1^{1/2} r \right) \right\|_{L^2(\Omega)}.$$
	
	Table~\ref{lapeig} shows the speed of this convergence and indicates that it is uniform.
	 \begin{table}[t]
	 	\caption{The convergence to the principal eigenvalue $\lambda_1 \approx 5.78318$ and to the corresponding eigenfunction $w_1$ computed
	 		by Algorithm~1 for the unit disc with $p=q=2$, $\alpha = 1$ and $\beta = 1$.  }
	 	\begin{center}
	 		\begin{tabular}{|c|c|c|c|}
	 			\hline	
	 			$k$ & $\lambda^k$    & $|\lambda_1 - \lambda^k|/\lambda_1$ & $\| u_k -w_1 \|_{L^\infty(\Omega)} $ \\ \hline
	 			0   & 10.0000        & 0.7291 & 1.0764  \\
	 			1   & 5.9232        & 0.0242 & 0.1405  \\
	 			2   & 5.7882        & 0.0008 & 0.0246  \\
	 			3   & 5.7834       & 0.0000 & 0.0045  \\
	 			4   & 5.7832        & 0.0000 & 0.0005  \\
	 			5   & 5.7832               & 0.0000   & 0.0001  \\
	 			\hline	
	 		\end{tabular}
	 		\label{lapeig}
	 	\end{center}
	 \end{table}
		

	 As a second test, we choose $p=q$ and consider large values of $p$. In this case system \eqref{mainsys} reduces to the scalar equation \eqref{plapeq} and it is known \cite[Lemma 11]{lind1} that $\lambda(p)^{1/p}$ converges to $\Lambda_\infty = 1$ when $p\rightarrow \infty$. We verify this fact numerically by applying Algorithm~1 to system \eqref{mainsyspolar} with $u_0=\cos(\pi r/2), v_0=\sin(\pi(r+1)/2)$.

	 The results are reported in Table \ref{scaleqlam} and confirm the expectations although the convergence is not as fast as in the previous test.
         Similarly, the corresponding eigenfunctions are known to converge to the distance function $d(x) = \operatorname{dist}(x,\partial\Omega)$. Figure \ref{cir-plapfig} shows the eigenfunctions computed by Algorithm~1 for three different values of $p$ and confirms the expected convergence.

	 \begin{table}[ht]
	 	\caption{ The principal eigenvalue $\lambda(p)$ of $p$-Laplacian in the unit circle computed by Algorithm~1 applied to system \eqref{mainsyspolar} with $p=q$.}
	 	\begin{center}
	 		
	 		\begin{tabular}{|l|l|l|c|}
	 			\hline	
	 			$p$   & $\lambda(p)$    & $\lambda(p)^{1/p}$  & $\| u(x) -d(x) \|_{L^\infty(\Omega)} $\\ \hline
	 			1.3   & 3.2660         & 2.5205             & 0.3864\\
	 			6     & 26.832         & 1.7301             & 0.4316\\
	 			18    & 166.02         & 1.3284             & 0.2605\\
	 			30    & 415.90         & 1.2226             & 0.1864\\
	 			100   & 4026.9         & 1.0865             & 0.0772\\
	 			200   & 15498          & 1.0494            & 0.0449\\
	 			300   & 34363          & 1.0354             & 0.0325\\
	 			400   & 60610          & 1.0279             & 0.0257\\
	 			\hline	
	 		\end{tabular}
	 		\label{scaleqlam}
	 	\end{center}
	 \end{table}

	 \begin{figure}[ht]
	 	\centering
	 	\subfloat[ $p=1.3$]{	\includegraphics[width=0.34\textwidth]{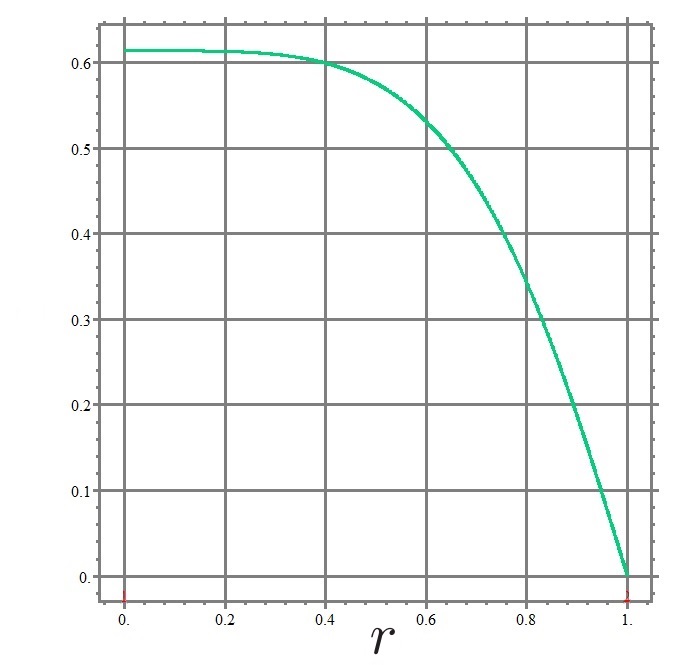}}
	 	\subfloat[$p=6$]{	\includegraphics[width=0.33\textwidth]{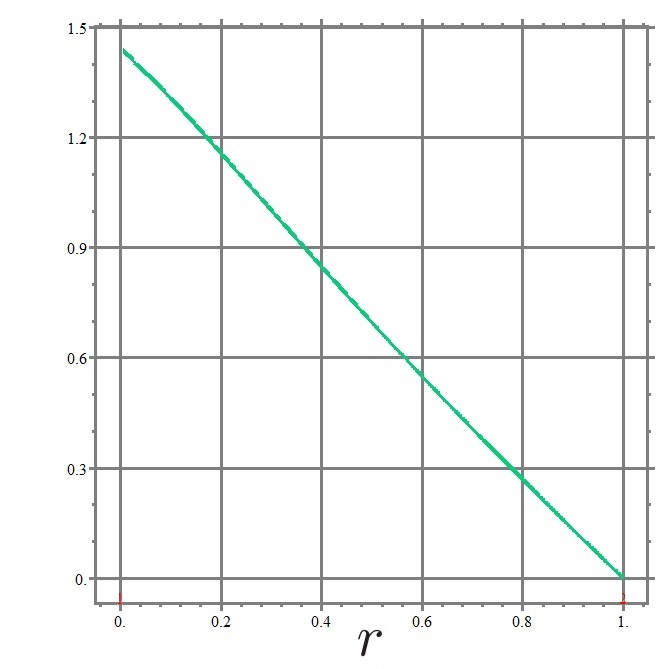}}
	 	\subfloat[$p=400$]{	\includegraphics[width=0.34\textwidth]{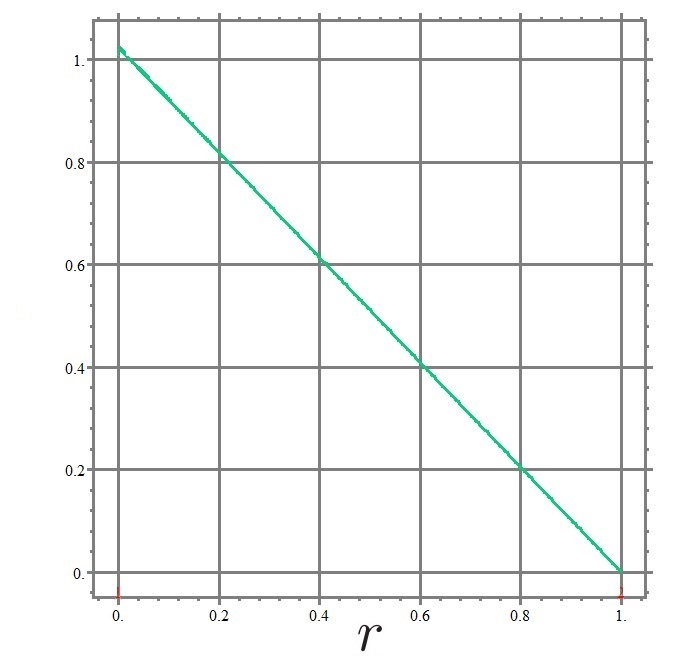} }
	 	\caption{Radial parts of eigenfunctions corresponding to the principal eigenvalue $\lambda(p)$ of $p$-Laplacian in the unit disc for ${p=1.3, 6, 400}$.
	 		 	}
	 	\label{cir-plapfig}
	 \end{figure}

	 In the third case, we consider $p$ not equal to $q$. The performance of Algorithm~1 applied to system \eqref{mainsyspolar}
	 with the initial guess $u_0=v_0=w_1$ is shown in Table~\ref{cirdifpq} for various values of $p$, $q$, $\alpha$, and $\beta$. As an example, the pair  of computed eigenfunction  $(u,v)$  corresponding to $\lambda (30,2)$ is presented in Figure \ref{cir-302-eigf}.
	
	 \begin{table}[h]
	 	\caption{ The principal eigenvalue $\lambda(p,q)$ of system \eqref{mainsys} in the unit circle computed by Algorithm~1 applied to system \eqref{mainsyspolar}. }
	 	\begin{center}
	 		
	 		\begin{tabular}{|l|l|l|l|l|l|}
	 			\hline	
	 			$p$   & $q$    & $\alpha$   & $\beta$  & $\lambda(p,q)$    \\ \hline
	 			30    & 1.5    &  1         & 1.4500   & 7.4486            \\
	 			30    & 2      &  1         & 1.9333   & 9.5034            \\
	 			30    & 5      &  1         & 4.8333   & 25.656            \\
	 			30    & 7      &  1         & 6.7666   & 40.194            \\
	 			30    & 10     &  1         & 9.6666   & 67.562            \\
	 			30    & 13     &  1         & 12.566   & 101.51            \\
	 			30    & 15     &  1         & 14.500   & 127.77            \\
	 			30    & 17     &  1         & 16.433   & 156.92            \\
	 			30    & 20     &  1         & 19.333   & 206.01            \\
	 			30    & 25     &  1         & 24.166   & 302.09             \\
	 			\hline	
	 		\end{tabular}
	 		\label{cirdifpq}
	 	\end{center}
	 \end{table}
	 \begin{figure}[ht]
	 	\centering
	 	\subfloat[ ]{	\includegraphics[width=0.4\textwidth]{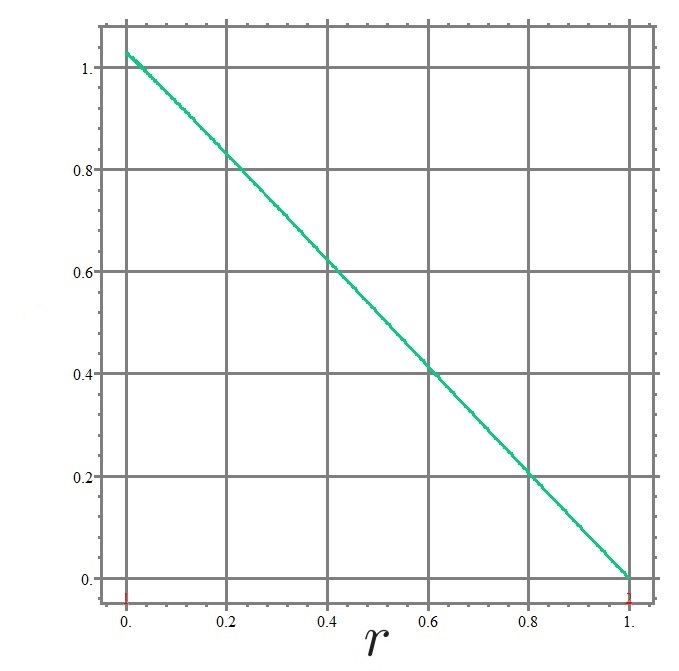}}
	 	\subfloat[]{	\includegraphics[width=0.4\textwidth]{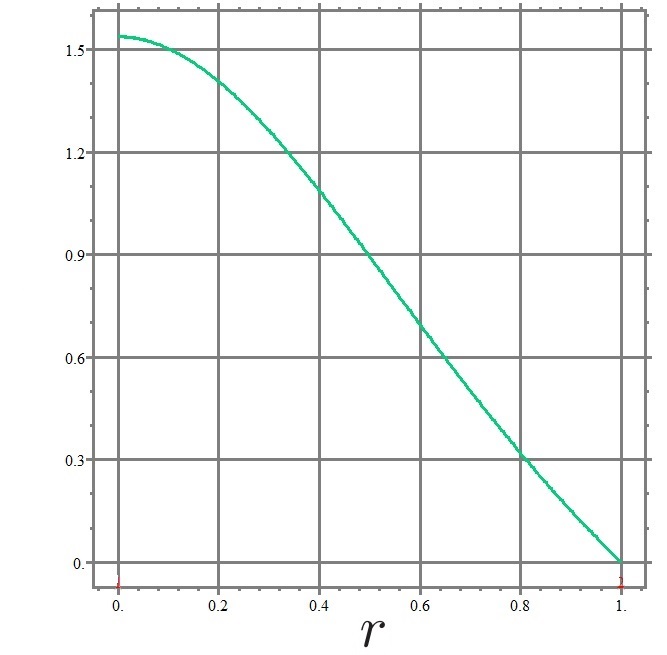}}
	 	\caption{{Radial parts of eigenfunctions $u$ (left) and $v$ (right) corresponding to $\lambda (30,2)$ with $\alpha=1$ and $\beta=1.9333$ for the unit disc.}}
	 	\label{cir-302-eigf}
	 \end{figure}
	\end{exam}

	 \begin{exam}\label{square-tri}
         In this example, we consider the square domain
	 $\Omega=\{(x_1,x_2)\in \mathbb{R}^2:\: 0< x_1< 2,\:  0< x_2< 2 \}$.
	 The principal eigenvalue of system \eqref{mainsys} corresponds to its main frequency and
	 we apply Algorithm 1 to determine it for different values of $p$ and $q$.
         Table~\ref{rectable} lists the resulting principal eigenvalues for a mesh with 2390 elements, 4904 nodes, and mesh size $h=1/16=0.0625$.
	 \begin{table}[h!]
	 	\caption{ Principal eigenvalues of \eqref{mainsys} in the square.}
	 	\begin{center}
	 		
	 		\begin{tabular}{|l|l|l|c|c|}
	 			\hline	
	 			$p$   & $q$    & $\alpha$  & $\beta$   & $\lambda(p,q)$       \\ \hline
	 			10    & 1.5    & 1         & 1.35      & 6.0294\\
	 			10    & 2      & 1         & 1.80      & 7.3695\\
	 			10    & 3      & 1         & 2.70      & 10.173\\
	 			10    & 4      & 1         & 3.60      & 13.183\\
	 			10    & 5      & 1         & 4.50      & 16.391\\
	 			10    & 6      & 1         & 5.60      & 19.788\\
	 			10    & 7      & 1         & 6.30      & 23.356\\
	 			10    & 8      & 1         & 7.20      & 27.076\\
	 			10    & 9      & 1         & 8.10      & 30.957\\
	 			10    & 10     & 1         & 9.00      & 34.999\\
	 			\hline	
	 		\end{tabular}
	 		\label{rectable}
	 	\end{center}
	 \end{table}
	 A typical pair of eigenfunctions $(u,v)$ computed by Algorithm~1 is illustrated in Figure \ref{receigf} for $p=10$, $q=2$, $\alpha=1$, and $\beta=1.8$.
	 \begin{figure}[h!]
	 	\centering
	 	\subfloat[ $u$]{	\includegraphics[width=0.4\textwidth]{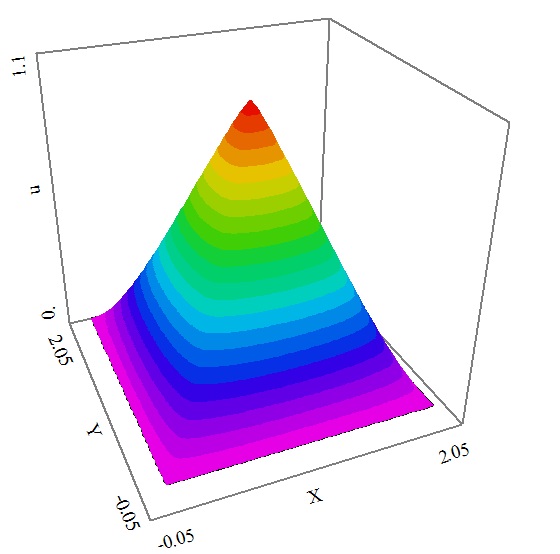}}
	 	\subfloat[$v$]{	\includegraphics[width=0.4\textwidth]{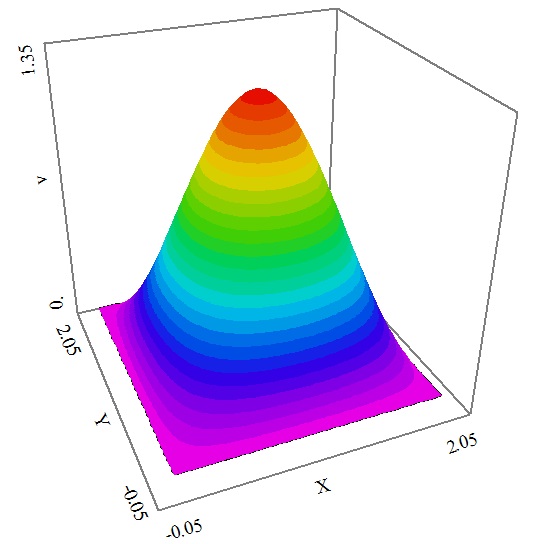}}
	 	\caption{{Eigenpair $(u,v)$  corresponding to $ \lambda (10,2)$ with ${\alpha=1, \: \beta=1.8}$.}}
	 	\label{receigf}
	 \end{figure}

	Further, to demonstrate the accuracy of computed approximations we test the order of convergence of the used finite element method. We solve the problem on a sequence of successively refined meshes and we compute the experimental order of convergence $EOC$ using the
	ratio of successive differences of eigenvalues for different mesh sizes as
	$$
	  EOC = \log_2\left|\frac{\lambda_h-\lambda_{\frac{h}{2}}}{\lambda_{\frac{h}{2}}-\lambda_{\frac{h}{4}}}\right|.
	$$
	We present it
	in Table~\ref{ratecon-p-10-q-5} for $p=10$, $\alpha=1$, $q=1.5, 5$, and $10$. The corresponding values of $\beta$ are given by \eqref{alfa} to be $1.35$, $4.50$, and $9$, respectively.
	 \begin{table}[h!]
	 \caption{{ Experimental orders of convergence $EOC$ for the square domain with $p=10$, $\alpha=1$, and $q=1.5, 5$, and $10$.
	}}
	 \begin{center}
	 \begin{tabular}{|r|c|c|c|c|c|c|c|}
	 	\hline
	 	$n$   & $h=\frac{1}{2^n}$ & $\lambda_{h}(10,1.5)$ & $EOC$ & $\lambda_{h}(10,5)$ & $EOC$ & $\lambda_{h}(10,10)$ & $EOC$ \\ \hline
	 	$0$   & 1               & 6.48002 & 3.2308 & 20.9681 & 2.0802 & 69.4507 & 2.1857\\
	 	$1$   & $\frac{1}{2}$   & 6.08031 & 2.5340 & 17.1136 & 2.1573 & 42.3979 & 2.3063\\
	 	$2$   & $\frac{1}{4}$   & 6.03774 & 2.9234 & 16.5020 & 3.7281 & 36.4513 & 2.2662\\
	 	$3$   & $\frac{1}{8}$   & 6.03039 &        & 16.4064 & 3.0514 & 35.2491 & 2.7465\\
	 	$4$   & $\frac{1}{16}$  & 6.02942 &        & 16.3910 &        & 34.9992 &       \\
	 	$5$   & $\frac{1}{32}$  &         &        & 16.3891 &        & 34.9619 &       \\	
	 	\hline
	 \end{tabular}
	 \label{ratecon-p-10-q-5}
	 \end{center}
	 \end{table}
        The expected order of convergence is two. The higher experimental orders of convergence observed in Table~\ref{ratecon-p-10-q-5} probably indicate the preasymptotic regime. On finer meshes the experimental order of convergence will probably decrease to values around two.

	\end{exam}
	 \begin{exam}\label{nonconvex}
	 Here we test Algorithm~1 for three other domains, namely for
	 an isosceles triangle, L-shaped domain, and a heart shaped domain.
	 Note that L-shaped and heart shaped domains are non-convex and singularities of eigenfunctions are expected in re-entrant corners.
	 To be more specific, the isosceles triangle has base 1 and altitude 1,
	 the L-shaped domain is $(0,3)^2\setminus[1,3]^2$, 	
	 and the heart shaped domain is $H=H_1\cup H_2\cup H_3$
	 where
	 \begin{align*}
	 	H_1&=\{(x,y)\in \mathbb{R}^2:\quad (x-1)^2+(y^2/4)<1, y\geq 0\},\\
	 	H_2&=\{(x,y)\in \mathbb{R}^2:\quad (x+1)^2+(y^2/4)<1, y\geq 0\},\\
	 	H_3&=\{(x,y)\in \mathbb{R}^2:\quad ({x^2}/{4})+({y^2}/{16})<1, y \leq 0\}.
	 \end{align*}

        Table \ref{tritable} lists the principal eigenvalues computed by Algorithm 1 for these three domains and different values of $p$ and $q$. The mesh size in all three cases was $h=1/16=0.0625$.
               For illustration we also present eigenfunctions $(u,v)$ corresponding to the principal eigenvalue $\lambda(3,10)$ and $\alpha=1$ and $\beta=6.6666$ in Figures~\ref{trieigf}--\ref{hearteigf} 
        for the isosceles triangle, L-shaped domain, and the heart shaped domain, respectively.

	 \begin{table}[ht]
	 \caption{ The principal eigenvalue for the isosceles triangle, L-shaped domain and heart shaped domain.}
	 \begin{center}	
	 \begin{tabular}{|c|c|c|c|c|c|c|}
	 \hline
	      &        &           &           & \multicolumn{3}{|c|}{$\lambda(p,q)$}\\
	 $p$  & $q$    & $\alpha$  & $\beta$   & triangle           & L-shape& heart \\ \hline
	 3    & 2      & 1         & 1.3333    & $7.9822\times10^1$ & 12.914 & 1.3330\\
	 3    & 3      & 1         & 2.0000    & $2.2725\times10^2$ & 23.632 & 1.1917\\
	 3    & 4      & 1         & 2.3333    & $6.1966\times10^2$ & 41.713 & 1.0268\\
	 3    & 5      & 1         & 3.3333    & $1.6384\times10^3$ & 71.810 & 0.8607\\
	 3    & 6      & 1         & 4.0000    & $4.2386\times10^3$ & 121.29 & 0.7061\\
	 3    & 7      & 1         & 4.6666    & $1.0778\times10^4$ & 201.73 & 0.5692\\
	 3    & 8      & 1         & 5.3333    & $2.7038\times10^4$ & 331.44 & 0.4523\\
	 3    & 9      & 1         & 6.0000    & $6.7066\times10^4$ & 539.05 & 0.3557\\
	 3    & 10     & 1         & 6.6666    & $1.6479\times10^5$ & 862.16 & 0.2766\\
	 \hline	
	 \end{tabular}
	 \label{tritable}
	 \end{center}
	 \end{table}

%
	 \begin{figure}[h]
	 	\centering
	 	\subfloat[ $u$]{	\includegraphics[width=0.4\textwidth]{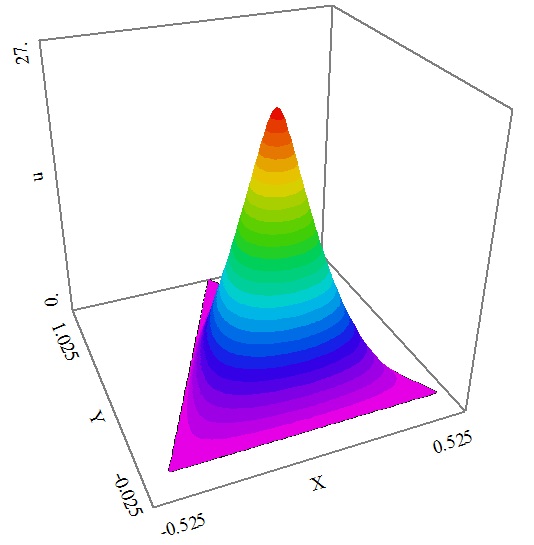}}
	 	\subfloat[$v$]{	\includegraphics[width=0.4\textwidth]{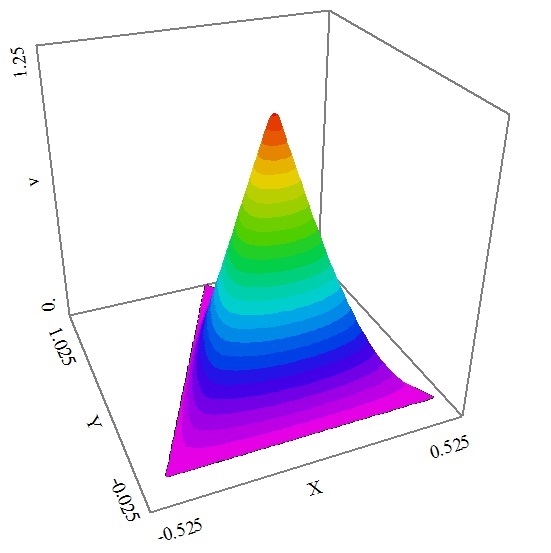}}
	 	\caption{{Eigenfunctions $u$ (left) and $v$ (right) corresponding to $\lambda (3,10)$ with ${\alpha=1}$ and $\beta=6.6666$ for the isosceles triangle.}}
	 	\label{trieigf}
	 \end{figure}
	 \begin{figure}[h]
	 	\centering
	 	\subfloat[ $u$]{	\includegraphics[width=0.4\textwidth]{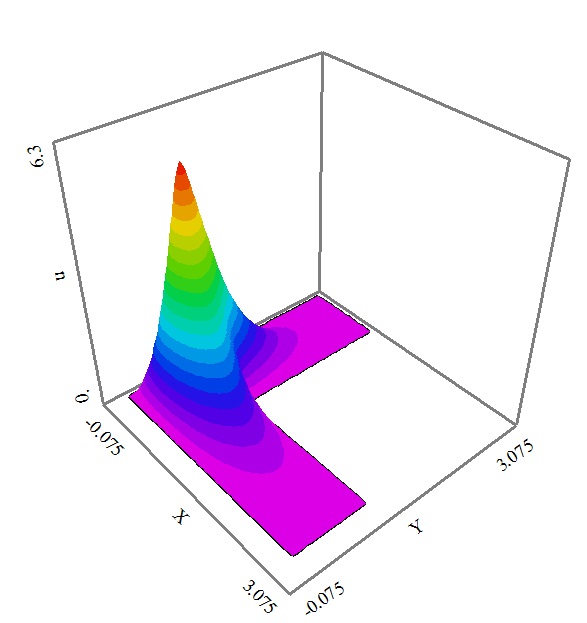}}
	 	\subfloat[$v$]{	\includegraphics[width=0.4\textwidth]{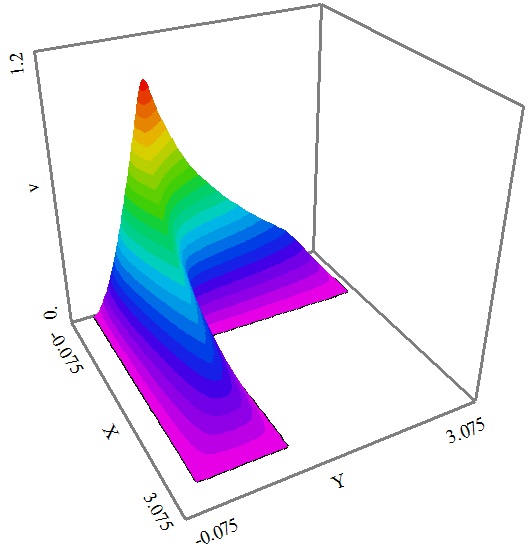}}
	 	\caption{{Eigenfunctions $u$ (left) and $v$ (right) corresponding to $\lambda (3,10)$ with ${\alpha=1}$ and $\beta=6.6666$ for the L-shaped domain.}}
	 	\label{lshapeeigf}
	 \end{figure}
	 \begin{figure}[h]
	 	\centering
	 	\subfloat[ $u$]{	\includegraphics[width=0.4\textwidth]{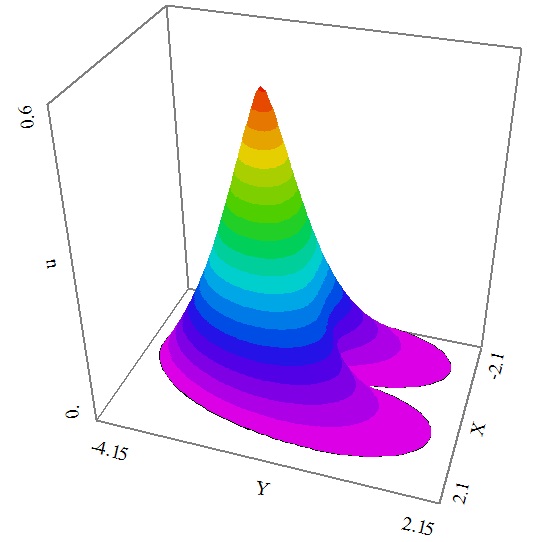}}
	 	\subfloat[$v$]{	\includegraphics[width=0.4\textwidth]{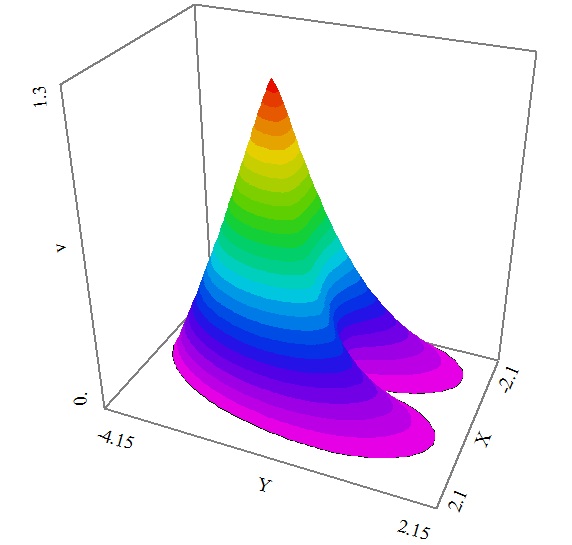}}
	 	\caption{{Eigenfunctions $u$ (left) and $v$ (right) corresponding to $ \lambda (3,10)$ with ${\alpha=1}$ and $\beta=6.6666$ for the heart shaped domain.}}
	 	\label{hearteigf}
	 \end{figure}
	\end{exam}

	  \begin{exam}\label{napoli}
	        In order to present the usage of Algorithm 1 for a more general quasilinear system as it was proposed at the end of Section~\ref{numeralgo}, we consider in this example a resonant quasilinear system of the following form
	  	\begin{equation}\label{napolisys}
	  	\left\{
	  	\begin{array}{lrl}
	  	- \Delta_{p} u =\Lambda(p,q) r(x) \alpha |u|^{\alpha-2} u |v|^{\beta}   &    \text{in } \Omega, \\
	  	- \Delta_{q} v =\Lambda(p,q) r(x) \beta |u|^{\alpha}  |v|^{\beta-2} v   &   \text{in } \Omega,\\
	  	u=v=0 &   \text{on } \partial\Omega,
	  	\end{array}
	  	\right.
	  	\end{equation}
	  	where $r\in L^\infty (\Omega)$ is a strictly positive function, $r(x) \geq m>0$.
	        This system has been studied intensively by several authors, see e.g. \cite{boccardo, bonder,Pezzo,napoli} to list just a few references.
	
	        The resonant quasilinear system \eqref{napolisys} differs from \eqref{mainsys} and has certain specific properties. For instance, in contrast to \eqref{mainsys}, the solution of system \eqref{napolisys} does not satisfy $u=v$ for $p=q$. However, Algorithm 1 with above mentioned generalizations can be successfully used to compute its principal eigenvalues and eigenfunctions.
	
	        For illustration we consider the square ${\Omega=\{(x_1,x_2)\in \mathbb{R}^2:\: 0< x_1< 2,\:  0< x_2< 2 \}}$ and function
	  		\begin{equation*}
	  		r(x_1, x_2)=	\left\{
	  			\begin{array}{lrl}
	  				1   &  \text{for }  0 < x_1 \leq 1, \\
	  				2   &  \text{for }  1 < x_1 < 2.\\
	  				  			\end{array}
	  			\right.
	  		\end{equation*}
	  	Principal eigenvalues $\Lambda(p,q)$ (corresponding to main frequencies) of system
	  	\eqref{napolisys} computed by the generalized Algorithm 1 for different values of $p,q,\alpha$ and $\beta$ are listed in Table~\ref{napolitable}. The used mesh is the same as in Example~\ref{square-tri}.
	  	  	\begin{table}[ht]
	  		\caption{The principal eigenvalue of the resonant quasilinear system \eqref{napolisys} on a square. }
	  		\begin{center}
	  			
	  			\begin{tabular}{|l|l|l|l|c|c|}
	  				\hline	
	  				$p$  & $q$    & $\alpha$  & $\beta$   & $\Lambda(p,q)$ & \text{Upper bound \eqref{upes}}       \\ \hline
	  				10    & 2      & 1         & 1.80      & 4.6239 & 24.1474\\
	  				10    & 3      & 1         & 2.70      & 4.3660 & 31.2932\\
	  				10    & 4      & 1         & 3.60      & 4.3459 & 32.9794\\
	  				10    & 5      & 1         & 4.50      & 4.4108 & 29.1125\\
	  				10    & 6      & 1         & 5.40      & 4.5119 & 22.1799\\
	  				10    & 7      & 1         & 6.30      & 4.6327 & 15.0333\\
	  				10    & 8      & 1         & 7.20      & 4.7612 & 9.4046\\
	  				10    & 9      & 1         & 8.10      & 4.8968 & 5.7214\\
	  				10    & 10     & 1         & 9.00      & 5.0362 & ---
	  				\\
	  				\hline	
	  			\end{tabular}
	  			\label{napolitable}
	  		\end{center}
	  	\end{table}
	  	
     We note that in \cite{boccardo,napoli}, the following upper bound on the first eigenvalue of  system \eqref{napolisys} with $p>q$ has been found:
     \begin{equation}\label{upes}
     \Lambda(p,q) \leq \frac{\Lambda(p)}{p}+\frac{m^{-1+q/p}}{q}\left(\frac{p}{q}\right)^q (\Lambda(p))^{q/p},
     \end{equation}
     where $\Lambda(p)$ is the first eigenvalue of the Dirichlet weighted $p$-Laplace eigenvalue problem
     \begin{equation}\label{wplap}
       - \Delta_{p} u =\Lambda(p) r(x) |u|^{p-2}  u  \quad    \text{in } \Omega.
     \end{equation}
     Numerical results presented in the last column of Table~\ref{napolitable} show that upper bound \eqref{upes} may considerably overestimate the true eigenvalue for some values of parameters $p,q,\alpha$ and $\beta$.

     Numerical results presented in Table~\ref{napolitable} also show that the lower bound derived in Theorem \ref{lowerest} for the first eigenvalue of \eqref{mainsys} cannot be straightforwardly generalized to system \eqref{napolisys}. For example, the value $\Lambda(10,4)=4.3459$ from Table~\ref{napolitable} is not above $\Lambda(4)=5.7534$ nor $\Lambda(10)=18.1873$.


	  \end{exam}


	 \section{Conclusions}
	 In this paper,  an elliptic eigenvalue system involving the $p$-Laplace operator has been considered. The principal eigenvalue and corresponding eigenfunctions of the system have been investigated both analytically and numerically.
	  We have provided an alternative proof for the simplicity of the principal eigenvalue and we have shown that this system reduces to the $p$-Laplace eigenvalue problem for a special choice of parameters.
	  Further, we developed a numerical algorithm in order to compute approximate principal eigenvalues and corresponding eigenfunctions. We showed how to generalize  this algorithm for gradient type systems.
	  The convergence of this algorithm was verified numerically for various examples, but
	  an analytical proof of convergence seems to be an interesting and difficult mathematical problem.
	

	\section*{Acknowledgements}
	{\footnotesize
	F. Bozorgnia   was  supported by the Portuguese National Science Foundation through FCT fellowships SFRH/BPD/33962/2009.
	  T. Vejchodsk\'y gratefully acknowledges the support of Neuron Fund for Support of Science, project no.~24/2016 and the institutional support RVO 67985840.
	}  

	 \providecommand{\bysame}{\leavevmode\hbox to3em{\hrulefill}\thinspace}
	 \providecommand{\MR}{\relax\ifhmode\unskip\space\fi MR }
	 \providecommand{\MRhref}[2]{%
	 	\href{http://www.ams.org/mathscinet-getitem?mr=#1}{#2}
	 }
	 \providecommand{\href}[2]{#2}

\end{document}